\documentclass[11pt,leqno]{amsart}
\usepackage{indentfirst, amssymb,amsmath,amsthm}
\usepackage{csquotes}
\theoremstyle{definition}
\newtheorem*{theoA}{Theorem A}
\newtheorem*{theoB}{Theorem B}
\newtheorem*{theoC}{Theorem C}
\newtheorem*{theoD}{Theorem D}
\newtheorem*{theoE}{Theorem E}
\newtheorem*{theoF}{Theorem F}
\newtheorem*{theoG}{Theorem G}
\newtheorem*{theoH}{Theorem H}

\newtheorem{theo}{Theorem}[section]
\newtheorem{lem}{Lemma}[section]

\newtheorem{defi}{Definition}[section]
\newtheorem{rem}{Remark}[section]
\newtheorem{question}{Question}[section]
\newcommand{\ol}{\overline}
\newcommand{\be}{\begin{equation}}
\newcommand{\ee}{\end{equation}}
\newcommand{\beas}{\begin{eqnarray*}}
\newcommand{\eeas}{\end{eqnarray*}}
\newcommand{\bea}{\begin{eqnarray}}
\newcommand{\eea}{\end{eqnarray}}
\newcommand{\lra}{\longrightarrow}
\numberwithin{equation}{section}
\theoremstyle{definitions}
\begin{document}
\title{}[Appeared in Tamkang Journal of Mathematics(TKJM)]
\vspace{3 cm}
\title[Uniqueness of a meromorphic function with its differential polynomial]{Uniqueness of power of a meromorphic function with its  differential polynomial}
\date{}
\author[B. Chakraborty ]{Bikash Chakraborty}
\date{}
\address{Department of Mathematics, Ramakrishna Mission Vivekananda Centenary College, Rahara, India-700 118}
\email{bikashchakraborty.math@yahoo.com, bikash@rkmvccrahara.org}
\maketitle
\let\thefootnote\relax
\begin{abstract}
In this paper, taking the question of Zhang and L\"{u} (\cite{11}) into the background, we present one theorem which will improve and extend some recent results related to the Br\"{u}ck Conjecture.
\end{abstract}
\section{Introduction}
Throughout this paper, we use the standard notations of the Nevanlinna theory of meromorphic functions as explained in (\cite{4}).\par
Let $f$  and $g$ be a two non-constant meromorphic functions defined in the open complex plane $\mathbb{C}$. If for some $a\in\mathbb{C}\cup\{\infty\}$, $f$ and $g$ have the same set of $a$-points with the same multiplicities, then we say that $f$ and $g$ share the value $a$ counting multiplicities (in short, CM) and if we do not consider the multiplicities, then $f$ and $g$ are said to share the value $a$ ignoring multiplicities (in short, IM). If $a=\infty$, then the zeros of $f-a$ means the poles of $f$.\par
A meromorphic function $a=a(z)(\not\equiv 0,~\infty)$ is called a small function with respect to $f$ provided that $T(r,a)=S(r,f)$ as $r\lra \infty, r\not\in E$, where $E$ is a set of positive real numbers with finite Lebesgue measure. If $a=a(z)$ is a small function, then we say that $f$ and $g$ share $a$ IM (resp. CM) according to $f-a$ and $g-a$ share $0$ IM (resp. CM).\par
The subject on sharing values between entire functions with its derivatives was first studied by Rubel and Yang (\cite{8}). In 1977, they proved the following result:
\begin{theoA}(\cite{8}) Let $f$ be a non-constant entire function. If $f$ and $f^{'}$ share two distinct finite numbers $a$, $b$ CM, then $f = f^{'}$.
\end{theoA}
In 1979, Mues and Steinmetz obtained the same result but in relax sharing environment as follows:
\begin{theoB}(\cite{7}) Let $f$ be a non-constant entire function. If $f$ and $f^{'}$ share two distinct values $a$, $b$ IM, then $f^{'}\equiv f$.
\end{theoB}
Subsequently, similar considerations have been made with respect to higher derivatives and more general differential expressions as well. Above theorems motivate researchers to study the relation between an entire function and its derivative counterpart for one CM shared value. In this direction, in 1996, the following famous conjecture was proposed by Br\"{u}ck (\cite{3}):\\

{\it {\bf Conjecture:} Let $f$ be a non-constant entire function such that the hyper order $\rho _{2}(f)$ of $f$ is not a positive integer or infinite, where
 $$\rho _{2}(f)=\limsup\limits_{r\lra\infty}\frac{ \log \log T(r,f)}{\log r}.$$ If $f$ and $f^{'}$ share a finite value $a$ CM, then $\frac{f^{'}-a}{f-a}=c$, where $c$ is a non-zero constant.}\\ \par
In recent years, many results have been published concerning the above conjecture, (see, \cite{bc1, bc2, bc3, th, th1, 1, 2a, 2, 5a, 8bc, ly}). Next we recall the following definitions:
\begin{defi}(\cite{5}) Let $p$ be a positive integer and $a\in\mathbb{C}\cup\{\infty\}$.
\begin{enumerate}
\item[(i)] $N(r,a;f\mid \geq p)$ (resp. $\ol N(r,a;f\mid \geq p)$) denotes the counting function (resp. reduced counting function) of those $a$-points of $f$ whose multiplicities are not less than $p$.
\item[(ii)] $N(r,a;f\mid \leq p)$ (resp. $\ol N(r,a;f\mid \leq p)$) denotes the counting function (resp. reduced counting function) of those $a$-points of $f$ whose multiplicities are not greater than $p$.
\end{enumerate}
\end{defi}
\begin{defi} (\cite{10}) For $a\in\mathbb{C}\cup\{\infty\}$ and a positive integer $p$, we define
$$N_{p}(r,a;f)=\ol N(r,a;f)+\ol N(r,a;f\mid \geq 2)+\ldots+\ol N(r,a;f\mid \geq p).$$
\end{defi}
\begin{defi} (\cite{10}) For $a \in \mathbb{C}\cup\{\infty\}$ and a positive integer p, we put
$$\delta_{p}(a,f)=1-\limsup\limits_{r \to \infty} \frac{{N}_{p}(r,a;f)}{T(r,f)}.$$
Thus $$0\leq \delta(a,f) \leq \delta_{p}(a,f) \leq \delta_{p-1}(a,f) \leq...\leq \delta_{2}(a,f) \leq \delta_{1}(a,f)=\Theta(a,f)\leq1.$$
\end{defi}
\begin{defi} (\cite{bc3}) For two positive integers $n$, $p$ we define
$$\mu_{p}= \min\{n,p\}~~\text{and}~~\mu_{p}^{*}= p+1-\mu_{p}.$$
Then clearly $$ N_{p}(r,0;f^{n}) \leq \mu_{p}N_{\mu_{p}^{*}}(r,0;f).$$
\end{defi}
\begin{defi} (\cite{1}) Let $z_{0}$ be a zero of $f-a$ of multiplicity $p$ and a zero of $g-a$ of multiplicity $q$.
 \begin{enumerate}
 \item [i)] We denote by $\ol N_{L}(r,a;f)$, the counting function of those $a$-points of $f$ and $g$ where $p>q\geq 1$,
 \item [ii)]  by $N^{1)}_{E}(r,a;f)$, we denote the counting function of those $a$-points of $f$ and $g$ where $p=q=1$ and
 \item [iii)]  by $\ol N^{(2}_{E}(r,a;f)$, we denote the counting function of those $a$-points of $f$ and $g$ where $p=q\geq 2$, each point in these counting functions is counted only once.
     \end{enumerate}
Similarly, we can define $\ol N_{L}(r,a;g),\; N^{1)}_{E}(r,a;g),\; \ol N^{(2}_{E}(r,a;g).$
\end{defi}
\begin{defi} (\cite{4a}) Let $k$ be a non-negative integer or infinity and $a\in\mathbb{C}\cup\{\infty\}$. By $E_{k}(a;f)$, we mean the set of all $a$-points of $f$, where an $a$-point of multiplicity $m$ is counted $m$ times if $m\leq k$ and $k+1$ times if $m>k$.\par If $E_{k}(a;f)=E_{k}(a;g)$, then we say that $f$ and $g$ share the value $a$ with weight $k$.
\end{defi}
Thus we note that $f$ and $g$ share a value $a-$ IM (resp. CM) if and only if $f$ and $g$ share $(a,0)$ (resp. $(a,\infty)$).\par
With the notion of weighted sharing of values Lahiri-Sarkar (\cite{5}) improved the result of Zhang (\cite{9}). In (\cite{10}), Zhang further extended the result of Lahiri-Sarkar (\cite{5}) and replaced the concept of value sharing by small function sharing.\par
In 2008, Zhang and L\"{u}(\cite{11}) further considered the uniqueness of the $n-$th power of a meromorphic function sharing a small function with its $k-$ th derivative and proved  the following theorem:
\begin{theoC}(\cite{11}) Let $ k(\geq 1)$, $n(\geq 1)$ be integers and $f$ be a non-constant meromorphic function. Also, let $a(z) (\not\equiv 0,\infty )$ be a small function with respect to $f$. Suppose $f^{n}-a$ and $f^{(k)}-a$ share $(0,l)$. If $l=\infty$ and \par
\be \label {e1.1}(3+k)\Theta(\infty,f)+2\Theta(0,f)+\delta_{2+k}(0,f) > 6+k-n, \ee\\
or, $l=0$ and \par
\be \label {e1.2}(6+2k)\Theta(\infty,f)+4\Theta(0,f)+2\delta_{2+k}(0,f) > 12+2k-n, \ee\\
then $f^{n}$ $\equiv$ $f^{(k)}$ .
\end{theoC}
In the same paper, Zhang and L\"{u} (\cite{11}) posed the following question:
\begin{question}
What will happen if  $f^{n}$ and $[f^{(k)}]^{s}$ share a small function?
\end{question}
In  2010, Chen and Zhang (\cite{2}) gave a answer to the above question, but unfortunately there were some gaps in the proof of the theorems in (\cite{2}). To rectify the gaps in (\cite{2}) as well as to answer the question of Zhang and L\"{u} (\cite{11}), in  2010, Banerjee and Majumder (\cite{1}) proved two theorems, one of which further improved {\it Theorem C} whereas the other answers the Question 1.1.
\begin{theoD} (\cite{1}) Let $ k(\geq 1)$, $n(\geq 1)$ be integers and $f$ be a non-constant meromorphic function. Also let $a(z) (\not\equiv 0,\infty )$ be a small function with respect to $f$. Suppose $f^{n}-a$ and $f^{(k)}-a$ share $(0,l)$. If $l\geq 2$ and \par
\be \label {e1.3} (3+k)\Theta(\infty,f)+2\Theta(0,f)+\delta_{2+k}(0,f) > 6+k-n, \ee\\
or, $l=1$ and \par
\be \label {e1.4} \left(\frac{7}{2}+k\right)\Theta(\infty,f)+\frac{5}{2}\Theta(0,f)+\delta_{2+k}(0,f) > 7+k-n, \ee\\
or, $l=0$ and \par
\be \label {e1.5}(6+2k)\Theta(\infty,f)+4\Theta(0,f)+\delta_{2+k}(0,f)+\delta_{1+k}(0,f) > 12+2k-n, \ee\\
then $f^{n}=f^{(k)}$ .
\end{theoD}
\begin{theoE} (\cite{1}) Let $k(\geq 1)$, $n(\geq 1)$, $m(\geq 2)$ be integers and $f$ be a non-constant meromorphic function.
Also, let $a(z)(\not\equiv 0,\infty)$ be a small function with respect to $f$. Suppose $f^{n}-a$ and $[f^{(k)}]^{m}-a$ share $(0,l)$.
If $l=2$ and\par
\be\label{e1.3a}(3+2k)\;\Theta (\infty,f)+2\;\Theta (0,f)+2\delta_{1+k}(0,f)> 7+2k-n,\ee\\
or,  $l=1$ and
 \be\label{e1.4a}\left(\frac{7}{2}+2k\right)\;\Theta (\infty,f)+\frac{5}{2}\;\Theta (0,f)+2\delta_{1+k}(0,f)> 8+2k-n,\ee\\
or, $l=0$ and
\be\label{e1.5a}(6+3k)\;\Theta (\infty,f)+4\;\Theta (0,f)+3\delta_{1+k}(0,f)> 13+3k-n,\ee\\ then $f^{n}\equiv [f^{(k)}]^{m}$.
\end{theoE}
It can be easily proved that {\it Theorem D} is a better result than {\it Theorem E} for $m=1$ case. Also, it is observed that in {\it Theorem E}, the conditions (\ref{e1.3a})-(\ref{e1.5a}) are independent of $m$. \par
Very recently, in order to improve the results of Zhang (\cite{10}), Li and Huang (\cite{5a}) obtained the following theorem.
\begin{theoF}(\cite{5a}) Let $ k(\geq 1)$, $l(\geq 0)$ be integers and $f$ be a non-constant meromorphic function. Also, let $a(z) (\not\equiv 0,\infty )$ be a small function with respect to $f$. Suppose $f-a$ and $f^{(k)}-a$ share $(0,l)$. If $l\geq 2$ and \par
\be \label {e1.6} (3+k)\Theta(\infty,f)+\delta_{2}(0,f)+\delta_{2+k}(0,f) > k+4,\ee\\
or, $l=1$ and \par
\be \label {e1.7} \left(\frac{7}{2}+k\right)\Theta(\infty,f)+\frac{1}{2}\Theta(0,f)+\delta_{2}(0,f)+\delta_{2+k}(0,f) > k+5, \ee\\
or, $l=0$ and \par
\be \label {e1.8} (6+2k)\Theta(\infty,f)+2\Theta(0,f)+\delta_{2}(0,f)+\delta_{1+k}(0,f)+\delta_{2+k}(0,f) > 2k+10, \ee\\
then $f \equiv f^{(k)}$.
\end{theoF}
In view of {\it Lemma \ref{l1.1}}, stated latter on, we see that {\it Theorem F} is better than {\it Theorem D} for $n=1$ case. Now, we recall the following definition.
\begin{defi} (\cite{4}) Let $n_{0j},n_{1j},\ldots,n_{kj}$ be non-negative integers. The expression $$M_{j}[f]=(f)^{n_{0j}}(f^{(1)})^{n_{1j}}\ldots(f^{(k)})^{n_{kj}}$$ is called a \emph{differential monomial} generated by $f$ of degree $d_{M_{j}}=d(M_{j})=\sum_{i=0}^{k}n_{ij}$ and weight $\Gamma_{M_{j}}=\sum_{i=0}^{k}(i+1)n_{ij}$. The sum
$$P[f]=\sum_{j=1}^{t}b_{j}M_{j}[f]$$ is called a \emph{differential polynomial} generated by $f$ of degree $\ol{d}(P)=\max\{d(M_{j}):1\leq j\leq t\}$
and weight $\Gamma=\Gamma_{P}=\max\{\Gamma_{M_{j}}:1\leq j\leq t\}$, where $T(r,b_{j})=S(r,f)$ for $j=1,2,\ldots,t$.\par

The numbers $\underline{d}(P)=\min\{d(M_{j}):1\leq j\leq t\}$ and $k$ (the highest order of the derivative of $f$ in $P[f]$) are called respectively the lower degree and order of $P[f]$.\par

The differential polynomial $P[f]$ is said to be homogeneous if $\ol{d}(P)$=$\underline{d}(P)$, otherwise $P[f]$ is called a non-homogeneous differential polynomial.\par
Also, we define $Q:=\max\; \{\Gamma _{M_{j}}-d(M_{j}): 1\leq j\leq t\}$; and for the sake of convenience for a differential monomial $M[f]$, we denote by  $\lambda =\Gamma_{M}-d_{M}$.
\end{defi}
Recently Charak and Lal (\cite{2a}) considered the possible extension of {\it Theorem D} in the direction of the question of Zhang and L\"{u} (\cite{11}) up to differential polynomial.
\begin{theoG} (\cite{2a}) Let $f$ be a non-constant meromorphic function and $n$  be a positive integer and $a(z) (\not\equiv 0,\infty )$ be a small function with respect to $f$. Let $P[f]$ be a non-constant differential polynomial in $f$. Suppose $f^{n}$ and $P[f]$ share $(a,l)$. If $l\geq 2$ and \par
\be \label {e1.9}(3+Q)\Theta(\infty,f)+2\Theta(0,f)+\ol{d}(P)\delta(0,f) > Q+5+2\ol{d}(P)-\underline{d}(P)-n, \ee\\
or, $l=1$ and \par
\be \label {e1.10} \left(\frac{7}{2}+Q\right)\Theta(\infty,f)+\frac{5}{2}\Theta(0,f)+\ol{d}(P)\delta(0,f) > Q+6+2\ol{d}(P)-\underline{d}(P)-n, \ee\\
or, $l=0$ and \par
\be \label {e1.11} (6+2Q)\Theta(\infty,f)+4\Theta(0,f)+2\ol{d}(P)\delta(0,f) > 2Q+4\ol{d}(P)-2\underline{d}(P)+10-n, \ee\\
then $f^{n} \equiv P[f]$.
\end{theoG}
Clearly, this is a supplementary result corresponding to {\it Theorem D} because by putting $P[f]=f^{(k)}$ in {\it Theorem G} one can't obtain {\it Theorem D}, rather in this case a set of stronger conditions are obtained  as a particular case of {\it Theorem F}. So the following question is natural:
\begin {question}
\label{q1}
\emph{Is it possible to improve {\it Theorem D} in the direction of {\it Theorem F} up to differential monomial so that the result give a positive answer to the question of Zhang and L\"{u}? }
\end{question}
To answer the above question, recently Banerjee and Chakraborty (\cite{bc3}) obtained the following Theorem:
\begin{theoH}(\cite{bc3}) Let $ k(\geq 1)$, $n(\geq 1)$ be integers and $f$ be a non-constant meromorphic function. Also, let $M[f]$ be a differential monomial of degree $d_{M}$ and weight $\Gamma_{M}$ and $k$ is the highest derivative in $M[f]$. Let $a(z) (\not\equiv 0,\infty )$ be a small function with respect to $f$. Suppose $f^{n}-a$ and $M[f]-a$ share $(0,l)$. If $l\geq 2$ and \par
\be \label{bce1.12} (3+\lambda)\Theta(\infty,f)+\mu_{2}\delta_{\mu_{2}^{*}}(0,f)+d_{M}\delta_{2+k}(0,f) > 3+\Gamma_{M}+\mu_{2}-n, \ee\\
or $l=1$ and \par
\be\label {bce1.13}  (\frac{7}{2}+\lambda)\Theta(\infty,f)+\frac{1}{2}\Theta(0,f)+\mu_{2}\delta_{\mu_{2}^{*}}(0,f)+d_{M}\delta_{2+k}(0,f) > 4+\Gamma_{M}+\mu_{2}-n, \ee\\
or $l=0$ and \par
\be \label {bce1.14} (6+2\lambda)\Theta(\infty,f)+2\Theta(0,f)+\mu_{2}\delta_{\mu_{2}^{*}}(0,f)+d_{M}\delta_{2+k}(0,f)+d_{M}\delta_{1+k}(0,f) > 8+2\Gamma_{M}+\mu_{2}-n, \ee\\
then $f^{n} \equiv M[f]$ .
\end{theoH}
In the same paper the following question was asked:
\begin{question}
Is it possible to extend \emph{Theorem H} up to differential polynomial instead of differential monomial?
\end{question}
To seek the possible answer of {\it Question 1.3} is the motivation of this paper.
\section{Main result}
\begin{theo}\label{t1} Let $ k(\geq 1)$, $n(\geq 1)$ be integers and $f$ be a non-constant meromorphic function. Let $P[f]$ be a \emph{homogeneous} differential polynomial of degree $\overline{d}(P)$ and weight $\Gamma_{P}$ such that $ \Gamma_{P}> (k+1) \underline{d}(P)- 2$, where $k$ is the highest derivative in $P[f]$. Also, let $a(z) (\not\equiv 0,\infty )$ be a small function with respect to $f$. Suppose $f^{n}-a$ and $P[f]-a$ share $(0,l)$. If $l\geq 2$ and\\
\bea \label{e1.12}   &&\left(\Gamma_{P}-\underline{d}(P)+3\right)\Theta(\infty,f)+\mu_{2}\delta_{\mu_{2}^{*}}(0,f)+\underline{d}(P)\delta_{_{2+\Gamma_{P} -\underline{d}(P)}}(0,f)\\
\nonumber&>& \Gamma_{P}+\mu_{2}+3-n,\eea\\
or, $l=1$ and
\bea\label {e1.13}  &&
\left(\Gamma_{P}-\underline{d}(P)+\frac{7}{2}\right)\Theta(\infty,f)+\frac{1}{2}\Theta(0,f)+\mu_{2}\delta_{\mu_{2}^{*}}(0,f)
+\underline{d}(P)\delta_{_{2+\Gamma_{P}-\underline{d}(P)}}(0,f)\\
\nonumber&>& \Gamma_{P}+\mu_{2}+4-n,\eea\\
or, $l=0$ and
\bea \label {e1.14} &&\left(2(\Gamma_{P} -\underline{d}(P))+6\right)\Theta(\infty,f)+2\Theta(0,f)+\mu_{2}\delta_{\mu_{2}^{*}}(0,f)
+\underline{d}(P)\delta_{_{1+\Gamma_{P} -\underline{d}(P)}}(0,f)\\
\nonumber&&+\underline{d}(P)\delta_{_{2+\Gamma_{P} -\underline{d}(P)}}(0,f)> 2\Gamma_{P} +\mu_{2}+8-n,\eea\\
then $f^{n} \equiv P[f]$.
\end{theo}
\begin{rem} If $P[f]$ be a non-constant differential monomial, then $\underline{d}(P)=\overline{d}(P)$. Thus our Theorem extends, generalizes Theorem H.
\end{rem}
From the above discussion, the following question is obvious:
\begin{question}
Is it possible to extend Theorem \ref{t1} up to an arbitrary differential polynomial?
\end{question}
\section{Lemmas} In this section, we present some lemmas which will be needed in this sequel. Let $F$, $G$ be two non-constant meromorphic functions and  $H$ be another meromorphic function which is defined as follows:
\be\label{e2.1} H=\left(\frac{\;\;F^{''}}{F^{'}}-\frac{2F^{'}}{F-1}\right)-\left(\frac{\;\;G^{''}}{G^{'}}-\frac{2G^{'}}{G-1}\right).\ee
\begin{lem}\label{l1.1}(\cite{bc3}) If $f$ is a non-constant meromorphic function, then
$$1+\delta_{2}(0,f) \geq 2\Theta(0,f).$$
\end{lem}
\begin{lem}\label{l11}(\cite{1}) If $F$ and $G$ share $(1,l)$, $\overline{N}(r,\infty;F)=\overline{N}(r,\infty;G)$ and $H\not\equiv 0$, then
\beas &&N(r,\infty;H)\\
  &\leq& \overline{N}(r,\infty;F)+\overline{N}(r,0;F|\geq 2)+\overline{N}(r,0;G|\geq 2)+\overline{N}_{0}(r,0;F')+\overline{N}_{0}(r,0;G')\\
 &&+\overline{N}_{L}(r,1;F)+\overline{N}_{L}(r,1;G)+S(r,F)+S(r,G). \eeas
 \end{lem}
\begin{lem}\label{l12} (\cite{bc3}) Let $F$ and $G$ share $(1,l)$. Then
$$\overline{N}_{L}(r,1;F)\leq \frac{1}{2}\overline{N}(r,\infty;F)+\frac{1}{2}\overline{N}(r,0;F)+S(r,F)~~\text{if}~~l\geq 1,$$
  and
$$\overline{N}_{L}(r,1;F)\leq \overline{N}(r,\infty;F)+\overline{N}(r,0;F)+S(r,F)~~\text{if}~~l=0.$$
Similar expressions also hold for $G$.
\end{lem}
\begin{lem}\label{l13}(\cite{bc3}) Let $F$ and $G$ share $(1,l)$ and $H \not\equiv 0$. Then \beas \overline{N}(r,1;F)+ \overline{N}(r,1;G) &\leq& N(r,\infty;H) + \overline{N}^{(2}_{E}(r,1;F)+\overline{N}_{L}(r,1;F)+\overline{N}_{L}(r,1;G)\\ && +\overline{N}(r,1;G)+S(r,F)+S(r,G).\eeas
\end{lem}
\begin{lem}\label{l14}
Let $f$ be a non-constant meromorphic function and $a(z)$ be a small function of $f$. Also, let $F=\frac{f^{n}}{a}$ and $G=\frac{P[f]}{a}$. If $F$ and $G$ share $(1,\infty)$, then one of the following cases hold:
\begin{enumerate}
\item[i)] $T(r) \leq N_{2}(r,0;F)+N_{2}(r,0;G)+\ol{N}(r,\infty;F)+\ol{N}(r,\infty;G)+\ol{N}_{L}(r,\infty;F)+\ol{N}_{L}(r,\infty;G)+S(r),$
\item[ii)]  $F\equiv G,$
\item[iii)] $FG\equiv 1$,
\end{enumerate}
where  $T(r)=\max\{T(r,F),T(r,G)\}$ and $S(r)=o(T(r))$, $r\in I$, $I$ is a set of infinite linear measure of $r\in(0,\infty)$.
\end{lem}
\begin{proof}
 Let $z_{0}$ be a pole of $f$ which is not a pole or zero of $a(z)$. Then $z_{0}$ is a pole of $F$ and $G$ simultaneously. Thus $F$ and $G$ share those pole of $f$ which is not zero or pole of $a(z)$.
Clearly
\beas N(r,\infty;H) &\leq& \ol{N}(r,0;F\geq2)+\ol{N}(r,0;G\geq2)+\ol{N}_{L}(r,\infty;F)+\ol{N}_{L}(r,\infty;G)\\
&& +\ol{N}_{0}(r,0;F')+\ol{N}_{0}(r,0;G')+S(r,f)\eeas
Rest of the proof can be carried out in the line of proof of Lemma 2.13 of (\cite{1.1}). So we omit the details.
\end{proof}
\begin{lem}\label{new} Let $p,~n$ be two positive integers. Then for $\varepsilon>0$
$$N_{p}(r,0;f^{n})\leq (n-n~\delta_{p}(0,f)+\varepsilon)T(r,f).$$
\end{lem}
\begin{proof} we see that
$$N_{p}(r,0;f^{n})\leq nN_{p}(r,0;f).$$
Rest part of the proof is obvious.
\end{proof}
\begin{lem}\label{l3}(\cite{8bc}) $N(r,\infty;P) \leq \overline{d}(P)N(r,\infty;f)+\left(\Gamma_P-\overline{d}(P)\right)\overline{N}(r,\infty;f).$
\end{lem}
\begin{lem}\label{l4}(\cite{ 6}) Let $f$ be a non-constant meromorphic function and let \[R(f)=\frac{\sum\limits _{i=0}^{n} a_{i}f^{i}}{\sum \limits_{j=0}^{m} b_{j}f^{j}}\] be an irreducible rational function in $f$ with constant coefficients $\{a_{i}\}$ and $\{b_{j}\}$ where $a_{n}\not=0$ and $b_{m}\not=0$. Then $$T(r,R(f))=pT(r,f)+S(r,f),$$ where $p=\max\{n,m\}$.
\end{lem}
\begin{lem} \label{l2.4}(\cite{new, 3ab}) Let $f$ be a meromorphic function and $P[f]$ be a differential polynomial. Then
$$ m\left(r,\frac{P[f]}{f^{\ol {d}(P)}}\right)\leq (\ol {d}(P)-\underline {d}(P)) m\left(r,\frac{1}{f}\right)+S(r,f).$$
\end{lem}
\begin{lem} \label{l5} (\cite{bc1, bc2}) Let $P[f]$ be a differential polynomial generated by a non-constant meromorphic function  $f$. Then
\beas N\left(r,\infty;\frac{P[f]}{f^{\ol {d}(P)}}\right)&\leq& (\Gamma _{P}-\ol {d}(P))\;\ol N(r,\infty;f)+(\ol {d}(P)-\underline {d} (P))\; N(r,0;f\mid\geq k+1)\\&&+Q \ol N(r,0;f\mid\geq k+1)+\ol {d}(P) N(r,0;f\mid\leq k)+S(r,f).\eeas
\end{lem}
\begin{lem}\label{l10}  Let $f$ be a non-constant  meromorphic function and $a(z)$ be a small function in $f$. Let us define $F=\frac{f^{n}}{a}, G=\frac{P[f]}{a}$. Then $FG \not\equiv 1$.\end{lem}
\begin{proof} On contrary, assume that $FG \equiv 1$, i.e., $P[f]f^{n}=(a(z))^2$. Then
$$N(r,0;f\mid\geq k+1)=S(r,f).$$
Now applying Lemmas \ref{l2.4}, \ref{l5} and the first fundamental theorem, we get
\begin{eqnarray*} &&(n+\overline{d}(P))T(r,f)\\
&=&T\left(r,\frac{P[f]}{f^{\overline{d}(P)}}\right)+S(r,f)\\
&\leq& (\ol {d}(P)-\underline {d}(P)) \left[T(r,f)-\{N(r,0;f\mid\leq k)+N(r,0;f\mid \geq k+1)\}\right]\\
&&+ (\Gamma _{P}-\ol {d}(P))\;\ol N(r,\infty;f)+(\ol {d}(P)-\underline {d}(P))\; N(r,0;f\mid\geq k+1)\\
&&+ Q\;\ol N(r,0;f\mid\geq k+1)+\ol {d}(P) N(r,0;f\mid\leq k)+S(r,f)\\
&\leq& (\ol {d}(P)-\underline {d}(P)) T(r,f)+\underline{d}(P) N(r,0;f\mid\leq k)+(\Gamma _{P}-\ol {d}(P))\;\ol N(r,\infty;f)+S(r,f)\\
&\leq& \ol {d}(P) T(r,f)+(\Gamma _{P}-\ol {d}(P))\;\ol N(r,\infty;P[f]f^{n})+S(r,f)\\
&\leq& \ol {d}(P) T(r,f)+(\Gamma _{P}-\ol {d}(P))\;\ol N(r,\infty;(a(z))^{2})+S(r,f)\\
&\leq& \ol {d}(P) T(r,f)+S(r,f),
\end{eqnarray*}
which is a contradiction.
\end{proof}
\begin{lem}\label{l8.5} For the differential polynomial $P[f]$,
 \beas  N(r,0;P[f]) &\leq& (\Gamma_{P} -\overline{d}(P))\ol{N}(r,\infty;f)+\underline{d}(P)\ol{N}(r,0;f)\\
 &+& (\overline{d}(P)-\underline{d}(P))\left(m(r,\frac{1}{f})+T(r,f)\right)+S(r,f).\eeas
\end{lem}
\begin{proof} From Lemma \ref{l2.4}, it is clear that
\bea\label{diwali1}\underline{d}(P)m(r,\frac{1}{f})\leq m(r,\frac{1}{P[f]})+S(r,f).\eea
Now using Lemmas \ref{l3}, \ref{l2.4} and inequality (\ref{diwali1}), we have
\beas  &&N(r,0;P[f])\\
&=&  T(r,P[f])-m(r,\frac{1}{P})+O(1)\\
&\leq& T(r,P[f])-\underline{d}(P)m(r,\frac{1}{f})+S(r,f)\\
&\leq& (\overline{d}(P)-\underline{d}(P))m(r,\frac{1}{f})+\overline{d}(P)m(r,f)+\overline{d}(P)N(r,\infty;f)\\
&&+\left(\Gamma_{P}-\overline{d}(P)\right)\overline{N}(r,\infty;f)- \underline{d}(P)m(r,\frac{1}{f})+S(r,f)\\
&\leq& \left(\Gamma_{P}-\overline{d}(P)\right)\overline{N}(r,\infty;f)+(\overline{d}(P)-\underline{d}(P))\left(m(r,\frac{1}{f})+T(r,f)\right)\\
&&+\underline{d}(P)\ol{N}(r,0;f)+S(r,f).
\eeas
Hence the proof is completed.
\end{proof}
\begin{lem}\label{l9.55} Let $j$ and $p$ be two positive integers satisfying $j\geq p+1$. Let $P[f]$ be a differential polynomial with $\Gamma_{P} > (k+1) \underline{d}(P)- (p+1)$. Then
$$\overline{N}_{_{(j+\Gamma_{P} -\underline{d}(P)}}(r,0;f^{\underline{d}(P)}) \leq \overline{N}_{(j}(r,0;P[f]).$$
\end{lem}
\begin{proof} Let $z_{0}$ be a zero of $f$ of order $t$. If $t~\underline{d}(P)<j+\Gamma_{P} -\underline{d}(P)$, then the proof is obvious. So we assume that $t~ \underline{d}(P)\geq j+\Gamma_{P} -\underline{d}(P)$. Now we consider two cases:\par
\textbf{Case-I} Let us assume that $t\geq k+1$. Then $z_0$ is a zero of $P[f]$ of order atleast
\beas &&\min\limits_{j}\{n_{0j}t+n_{1j}(t-1)+\ldots+n_{kj}(t-k)\}\\
&=&\min\limits_{j}\{t dM_{j}- (\Gamma_{M_{j}}-dM_{j})\}\\
&=&(t+1)\underline{d}(P)-\max\limits_{j}\{\Gamma_{M_{j}}\}\\
&\geq& (j+\Gamma_{P} -\underline{d}(P))+\underline{d}(P)-\Gamma_{P} \geq j.
\eeas
So the proof is clear.\par
\textbf{Case-II} Next we us assume that $t\leq k$. Then
\beas k~\underline{d}(P) &\geq& t \underline{d}(P) \geq j+\Gamma_{P} -\underline{d}(P)\\
&\geq& p+1+\Gamma_{P} -\underline{d}(P),
\eeas
which is a contradiction as $ \Gamma_{P} > (k+1) \underline{d}(P)- (p+1)$.
\end{proof}
\begin{lem}\label{l9} Let $j$ and $p$ be two positive integer satisfying $j\geq p+1$. Let $P[f]$ \emph{homogeneous differential polynomial} with $\Gamma_{P} > (k+1) \underline{d}(P)- (p+1)$. Then
$$N_{p}(r,0;P[f]) \leq N_{_{p+\Gamma_{P} -\underline{d}(P)}}(r,0;f^{\underline{d}(P)})+(\Gamma_{P} -\underline{d}(P))\overline{N}(r,\infty;f)+S(r,f).$$
\end{lem}
\begin{proof}
From Lemmas \ref{l8.5}, \ref{l9.55}, we have
\beas &&N_{p}(r,0;P[f])\\
&\leq& (\Gamma_{P} -\overline{d}(P))\ol{N}(r,\infty;f)+\ol{N}(r,0;f^{\underline{d}(P)})-\sum\limits_{j=p+1}^{\infty}\ol{N}_{(j}(r,0,P[f])+S(r,f)\\
&\leq& (\Gamma_{P} -\overline{d}(P))\ol{N}(r,\infty;f)+N_{_{p+\Gamma_{P} -\underline{d}(P)}}(r,0;f^{\underline{d}(P)})\\
&&+\sum\limits_{j=p+\Gamma_{P} -\underline{d}(P)+1}^{\infty}\ol{N}_{(j}(r,0;f^{\underline{d}(P)})-\sum\limits_{j=p+1}^{\infty}\ol{N}_{(j}(r,0;P[f])+S(r,f)\\
&\leq& (\Gamma_{P} -\overline{d}(P))\ol{N}(r,\infty;f)+N_{_{p+\Gamma_{P} -\underline{d}(P)}}(r,0;f^{\underline{d}(P)})+S(r,f).
\eeas
This completes the proof.
\end{proof}
\section {Proof of the theorem}
\begin{proof} Suppose that $$F=\frac{f^{n}}{a(z)}~~\text{and}~~~G=\frac{P[f]}{a(z)}.$$ Then $F-1=\frac{f^{n}-a(z)}{a(z)}$, $G-1=\frac{P[f]-a(z)}{a(z)}$. Since $f^{n}$ and $P[f]$ share $(a,l)$, it follows that $F$ and $G$ share $(1,l)$ except the zeros and poles of $a(z)$. Now we consider the following two cases.\\
{\bf Case 1.} First we assume that $H\not\equiv 0$.\\
\textbf{Subcase-1.1.} If $l\geq 1$, then using the second fundamental theorem and Lemmas \ref{l13} and \ref{l11}, we get
\bea\nonumber &&T(r,F)+T(r,G)\\
\nonumber &\leq& \overline{N}(r,\infty;F)+\overline{N}(r,\infty;G)+\overline{N}(r,0;F)+\overline{N}(r,0;G)+N(r,\infty;H) \\
&& \nonumber+ \overline{N}^{(2}_{E}(r,1;F)+\overline{N}_{L}(r,1;F)+\overline{N}_{L}(r,1;G)+\overline{N}(r,1;G)\\
&& \nonumber
-\overline{N}_{0}(r,0;F^{'})-\overline{N}_{0}(r,0;G^{'})+S(r,f)\\
& \label{t2} \leq& 2\overline{N}(r,\infty;F)+\overline{N}(r,\infty;G)+N_{2}(r,0;F)+N_{2}(r,0;G) +\overline{N}^{(2}_{E}(r,1;F)\\
&& \nonumber +2\overline{N}_{L}(r,1;F)+2\overline{N}_{L}(r,1;G)+\overline{N}(r,1;G) +S(r,f).
\eea
\textbf{Subcase-1.1.1.} If $l\geq 2$, then using the inequality (\ref{t2}), we get
\beas &&T(r,F)+T(r,G)\\
 & \leq& 2\overline{N}(r,\infty;F)+\overline{N}(r,\infty;G)+N_{2}(r,0;F)+N_{2}(r,0;G) +\overline{N}^{(2}_{E}(r,1;F)\\
&&+2\overline{N}_{L}(r,1;F)+2\overline{N}_{L}(r,1;G)+\overline{N}(r,1;G)+S(r,f)\\
& \leq& 2\overline{N}(r,\infty;F)+\overline{N}(r,\infty;G)+\mu_{2}N_{\mu_{2}^{*}}(r,0;f)+N_{2}(r,0;G)+ N(r,1;G)\\
&& +S(r,f).\eeas
i.e., for any $\varepsilon > 0$, in view of Lemma \ref{l9}, the above inequality becomes
\beas && n~T(r,f)\\
&\leq & (\Gamma_{P}-\underline{d}(P)+3)\overline{N}(r,\infty;f)+\mu_{2}N_{\mu_{2}^{*}}(r,0;f)+N_{_{2+\Gamma_{P}-\underline{d}(P)}}(r,0;f^{\underline{d}(P)})+S(r,f)\\
& \leq & \{(\Gamma_{P}-\underline{d}(P)+3)-(\Gamma_{P}-\underline{d}(P)+3)\Theta(\infty,f)+\mu_{2}-\mu_{2}\delta_{\mu_{2}^{*}}(0,f)\\
&& +\underline{d}(P)-\underline{d}(P)\delta_{_{2+\Gamma_{P}-\underline{d}(P)}}(0,f)+\varepsilon\}T(r,f)+S(r,f).
\eeas
i.e.,
\beas &&(\Gamma_{P}-\underline{d}(P)+3)\Theta(\infty,f)+\mu_{2}\delta_{\mu_{2}^{*}}(0,f)+\underline{d}(P)\delta_{_{2+\Gamma_{P}-\underline{d}(P)}}(0,f)\leq \Gamma_{P}+\mu_{2}+3-n,\eeas
which contradicts to the condition (\ref{e1.12}) of Theorem \ref{t1}.\par
\textbf{Subcase-1.1.2.} If $l=1$, then using the inequality (\ref{t2}) and  Lemma \ref{l12}, we get
\beas && T(r,F)+T(r,G)\\
&\leq& 2\overline{N}(r,\infty;F)+\overline{N}(r,\infty;G)+N_{2}(r,0;F)+N_{2}(r,0;G) +\overline{N}^{(2}_{E}(r,1;F)\\
&& +2\overline{N}_{L}(r,1;F)+2\overline{N}_{L}(r,1;G)+\overline{N}(r,1;G)+S(r,f)\\
& \leq& \frac{5}{2}\overline{N}(r,\infty;F)+\overline{N}(r,\infty;G)+\frac{1}{2}\overline{N}(r,0;F)+\mu_{2}N_{\mu_{2}^{*}}(r,0;f)+N_{2}(r,0;G)\\
&&+\overline{N}^{(2}_{E}(r,1;F)+\overline{N}_{L}(r,1;F)+2\overline{N}_{L}(r,1;G)+\overline{N}(r,1;G)+S(r,f)\\
& \leq& \frac{5}{2}\overline{N}(r,\infty;F)+\overline{N}(r,\infty;G)+\frac{1}{2}\overline{N}(r,0;F)+\mu_{2}N_{\mu_{2}^{*}}(r,0;f)+N_{2}(r,0;G)\\
&&+N(r,1;G)+S(r,f).
\eeas
i.e., for any $\varepsilon > 0$, in view of Lemma \ref{l9}, the above inequality becomes
\beas && n~T(r,f)\\
&\leq& (\Gamma_{P}-\underline{d}(P)+\frac{7}{2})\overline{N}(r,\infty;f)+\frac{1}{2}\overline{N}(r,0;f)+\mu_{2}N_{\mu_{2}^{*}}(r,0;f)\\
&&+ N_{_{2+\Gamma_{P}-\underline{d}(P)}}(r,0;f^{\underline{d}(P)})+S(r,f)\\
&\leq& \{(\Gamma_{P}-\underline{d}(P)+\frac{7}{2})-(\Gamma_{P}-\underline{d}(P)+\frac{7}{2})\Theta(\infty,f)+\frac{1}{2}-\frac{1}{2}\Theta(0,f)+\mu_{2}\\
&&-\mu_{2}\delta_{\mu_{2}^{*}}(0,f)+\underline{d}(P)-\underline{d}(P)\delta_{_{2+\Gamma_{P}-\underline{d}(P)}}(0,f)+\varepsilon\}T(r,f)+S(r,f).\eeas
i.e.,
\beas &&(\Gamma_{P}-\underline{d}(P)+\frac{7}{2})\Theta(\infty,f)+\frac{1}{2}\Theta(0,f)+\mu_{2}\delta_{\mu_{2}^{*}}(0,f)
+\underline{d}(P)\delta_{_{2+\Gamma_{P}-\underline{d}(P)}}(0,f)\\
&\leq& \Gamma_{P}+\mu_{2}+4-n,\eeas
which contradicts to the condition (\ref{e1.13}) of Theorem \ref{t1}.\\
\textbf{Subcase-1.2.} If $l=0$, then applying the second fundamental theorem and Lemmas \ref{l11}, \ref{l12}, \ref{l13}, we get
\bea \nonumber &&T(r,F)+T(r,G)\\
\nonumber&\leq& \overline{N}(r,\infty;F)+\overline{N}(r,0;F)+\overline{N}(r,1;F)+\overline{N}(r,\infty;G) +\overline{N}(r,0;G)\\
\nonumber && +\overline{N}(r,1;G)-\overline{N}_{0}(r,0;F')-\overline{N}_{0}(r,0;G')+S(r,F)+S(r,G)\\
\nonumber & \leq & \overline{N}(r,\infty;F)+\overline{N}(r,0;F)+\overline{N}(r,\infty;G)+\overline{N}(r,0;G)+N(r,\infty;H)\\
\nonumber && +\overline{N}^{(2}_{E}(r,1;F)+\overline{N}_{L}(r,1;F)+\overline{N}_{L}(r,1;G)+\overline{N}(r,1;G)\\
\nonumber && -\overline{N}_{0}(r,0;F')-\overline{N}_{0}(r,0;G')+S(r,F)+S(r,G)\\
\nonumber & \leq& 2\overline{N}(r,\infty;F)+\overline{N}(r,\infty;G)+N_{2}(r,0;F)+N_{2}(r,0;G)+\overline{N}^{(2}_{E}(r,1;F)\\
\nonumber && +2\overline{N}_{L}(r,1;F)+2\overline{N}_{L}(r,1;G)+\overline{N}(r,1;G)+S(r,f) \\
\nonumber & \leq & 2\overline{N}(r,\infty;F)+\overline{N}(r,\infty;G)+\mu_{2}N_{\mu_{2}^{*}}(r,0,f)+N_{2}(r,0;G)\\
\nonumber && +2(\overline{N}(r,\infty;F)+\overline{N}(r,0;F))+\overline{N}(r,\infty;G)+\overline{N}(r,0;G)\\
\nonumber && +\overline{N}^{(2}_{E}(r,1;F)+\overline{N}_{L}(r,1;G)+\overline{N}(r,1;G)+S(r,f)\\
&\label{r11} \leq&  4\overline{N}(r,\infty;F)+\mu_{2}N_{\mu_{2}^{*}}(r,0,f)+N_{2}(r,0;G)+2\overline{N}(r,\infty;G)\\
\nonumber && +\overline{N}(r,0;G)+2\overline{N}(r,0;F)+T(r,G)+S(r,f)\eea
i.e., for any $\varepsilon > 0$, in view of Lemma \ref{l9}, the above inequality becomes
\beas && n~T(r,f)\\
&\leq& \left(2(\Gamma_{P}-\underline{d}(P))+6\right)\overline{N}(r,\infty;f)+2\overline{N}(r,0;f)+\mu_{2}N_{\mu_{2}^{*}}(r,0,f)\\
&&+N_{_{2+\Gamma_{P}-\underline{d}(P)}}(r,0;f^{\underline{d}(P)})+N_{_{1+\Gamma_{P}-\underline{d}(P)}}(r,0;f^{\underline{d}(P)})+S(r,f)\\
& \leq& \{(2(\Gamma_{P}-\underline{d}(P))+6)-(2(\Gamma_{P}-\underline{d}(P))+6)\Theta(\infty,f)+2-2\Theta(0,f)\\
&&+\mu_{2}-\mu_{2}\delta_{\mu_{2}^{*}}(0,f)+\underline{d}(P)-\underline{d}(P)\delta_{_{1+\Gamma_{P}-\underline{d}(P)}}(0,f)+\underline{d}(P)\\
&&-\underline{d}(P)\delta_{_{2+\Gamma_{P}-\underline{d}(P)}}(0,f)+\varepsilon\}T(r,f)+S(r,f).\eeas
i.e.,
\beas &&(2(\Gamma_{P}-\underline{d}(P))+6)\Theta(\infty,f)+2\Theta(0,f)+\mu_{2}\delta_{\mu_{2}^{*}}(0,f)
+\underline{d}(P)\delta_{_{1+\Gamma_{P}-\underline{d}(P)}}(0,f)\\
&&+\underline{d}(P)\delta_{_{2+\Gamma_{P}-\underline{d}(P)}}(0,f)\leq 2\Gamma_{P}+\mu_{2}+8-n,\eeas
which  contradicts to the condition (\ref{e1.14}) of Theorem \ref{t1}.\par

{\bf Case 2.} Next we assume that $H\equiv 0$. Then on integration of (\ref{e2.1}), we get,
\begin{equation}\label{3.3} \frac{1}{G-1}\equiv\frac{A}{F-1}+B,
\end{equation}
 where $A(\neq 0)$ and $B$ are complex constants. Clearly $F$ and $G$ share $(1,\infty)$. Also, by construction of $F$ and $G$, $F$ and $G$ share $(\infty,0)$. So using Lemma \ref{l9} and condition (\ref{e1.12}) of Theorem \ref{t1}, we obtain
\beas && N_{2}(r,0;F)+N_{2}(r,0;G)+\ol{N}(r,\infty;F)+\ol{N}(r,\infty;G)+\ol{N}_{L}(r,\infty;F)+\ol{N}_{L}(r,\infty;G)\\
&&+S(r)\\
&\leq& \mu_{2}N_{\mu_{2}^{*}}(r,0;f)+N_{_{2+\Gamma_{P}-\underline{d}(P)}}(r,0;f^{\underline{d}(P)})+(\Gamma_{P}-\underline{d}(P)+3)\overline{N}(r,\infty;f)+S(r)\\
&\leq&\{\left(\Gamma_{P}+\mu_{2}+3\right)-((\Gamma_{P}-\underline{d}(P)+3)\Theta(\infty,f)-\mu_{2}\delta_{\mu_{2}^{*}}(0,f)\\
&&-\underline{d}(P)\delta_{_{2+\Gamma_{P}-\underline{d}(P)}}(0,f)+\varepsilon\}T(r,f)+S(r)\\
&<& T(r,F)+S(r),\eeas
where $\varepsilon > 0$ is any small quantity. Hence using Lemma \ref{l14} and Lemma \ref{l10}, we can conclude that  $F\equiv G$, i.e., $$f^{n} \equiv P[f].$$ Hence the proof is completed.
\end{proof}
\begin{center} {\bf Acknowledgement} \end{center}
The author is grateful to the referee for his valuable suggestions which considerably improved the presentation of the paper. The author also acknowledge the kind advice of Prof. Nan Li for the improvement of this paper.

\end{document}